\documentclass[a4paper,10pt]{article}

\usepackage[english]{babel}
\usepackage[T1]{fontenc}
\usepackage[utf8]{inputenc}
\usepackage{units}
\usepackage[margin=1.25in]{geometry}

\usepackage{amsmath}
\usepackage{afterpage}

\usepackage{graphicx}
\usepackage{subfigure}
\usepackage{wrapfig}
\usepackage{indentfirst}
\usepackage{latexsym}
\usepackage{sidecap}
\usepackage{booktabs}
\usepackage{amsthm}
\usepackage{amssymb}

\usepackage{setspace}

\usepackage{mathrsfs}
\usepackage{textcomp}
\usepackage{braket}
\usepackage{bbm}
\usepackage{colortbl}
\usepackage{bbm}
\usepackage{comment}
\usepackage[colorinlistoftodos]{todonotes}

\theoremstyle{plain}
\newtheorem{Theorem}{Theorem}[section]
\theoremstyle{definition}
\newtheorem{Definition}{Definition}[Theorem]
\theoremstyle{remark}
\newtheorem{Remark}[Theorem]{Remark}
\theoremstyle{remark}

\theoremstyle{plain}

\theoremstyle{plain}

\theoremstyle{plain}
\newtheorem{Proposition}[Theorem]{Proposition}
\theoremstyle{remark}
\newtheorem{Example}{Example}[section]
\theoremstyle{remark}

\theoremstyle{remark}

\theoremstyle{remark}
\newtheorem{Assumptions}[Theorem]{Assumptions}

\newcommand\RR{\mathbb{R}}
\newcommand\NN{\mathbb{N}}

\newcommand{\XX}{\mathcal{X}^2}

\usepackage{comment}
\usepackage[colorinlistoftodos]{todonotes}

\newcommand{\Ind}[1]{\mathbbm{1}_{\left [#1\right ]}}

\newcommand{\GG}{\mathbb{G}}
\newcommand{\uu}{\mathbf{u}}
\newcommand{\vv}{\mathbf{v}}
\newcommand{\XXX}{\mathbf{X}}
\newcommand{\YYY}{\mathbf{Y}}

\newcommand{\AAA}{\mathcal{A}}

\newcommand{\Ver}{\mathrm{v}}
\newcommand{\TT}{\mathcal{T}(t)}
\newcommand{\EP}{\mathcal{E}^2}

\usepackage{fancyhdr}
\usepackage{calc} 
\begin{document}

\date{}

\title{Stochastic reaction-diffusion equations on networks with dynamic time-delayed boundary conditions}

\author{Francesco Cordoni$^1$ and Luca Di Persio$^2$\\
$^1$ University of Trento - Department of Mathematics \\ Via Sommarive, 14 Trento, 38123-ITALY\\
e-mail: francesco.cordoni@unitn.it
\\[2pt]
$^2$ University of Verona - Department of Computer Science\\ Strada le Grazie, 15 Verona, 37134-ITALY\\
e-mail: luca.dipersio@univr.it
}

\maketitle

\begin{abstract}

We consider a reaction-diffusion equation on a network subjected to dynamic boundary conditions, with time  delayed behaviour, also allowing for  multiplicative Gaussian noise perturbations. Exploiting semigroup theory, we  rewrite the aforementioned stochastic problem as an abstract stochastic partial differential equation taking values in a suitable product Hilbert space, for which we prove the existence and uniqueness of a mild solution. Eventually, a stochastic optimal control application is studied.

\end{abstract}

\section{Introduction}\label{SEC:Intro}

Recent years have seen an increasing attention to the study of diffusion problems on networks, especially in connection with the theory of stochastic processes. In fact, there is a broad area of possible applications where the mathematical use  of graphs and random dynamics stated on them, play a crucial role, as in the case, e.g., of quantum mechanics, see, e.g. \cite{Tum}, the books \cite{GarZoll, Kleinert} and references therein;  in  neurobiology, as an example concerning the study of stochastic system of the \textit{FitzHugh-Nagumo} type, see, e.g., \cite{ALbDP, Allen, BarbuCorDP,BZM,Car}; or in finance, see, e.g., \cite{BenazzoliDP,CDPBackward1,CDPBackward2,HiSa} and references therein, particularly in the light of numerical applications, see, e.g., \cite{DPBonolloPellegrini}

Concerning the aforementioned ambit, a possible approach which has shown to   be particularly useful, is to introduce a suitable infinite dimensional space of functions that takes into account the underlying graph domain and then tackle the diffusion problem exploiting both functional analytic tools and infinite dimensional analysis. This technique had led to a systematic study of Stochastic Partial Differential Equations (SPDEs) on networks, showing that it is in general possible to rewrite a diffusion problem defined on  a network in a general abstract form, see, e.g., \cite{BZM,Car,Car3, DPZ}, and the monograph \cite{MugB} for a detailed introduction to the subject.

One of the main issues that appears in rewriting the initial problem into an operatorial abstract setting, is to choose the right boundary conditions (BC), that the diffusion problem has to satisfy. In order to overcome the latter, a systematic study of abstract SPDE equipped with different possible BC has been carried up during last years. The typical conditions when one has to deal with diffusion problems governed by a second order differential operator are the so-call \textit{generalized Kirchhoff conditions}, see, e.g., \cite{Mug}. Nevertheless rather recently, different types of general BC has been proposed, such as non-local BC, allowing for non-local interaction of non-adjacent vertex of the graph, see, e.g., \cite{Car,DPZ}, or dynamic BC, see, e.g., \cite{BZM,MugR}, or also mixed type BC, allowing for both static and dynamic non-local boundary conditions, see, e.g., \cite{CDP}.

In the present work we consider a new type of non-local BC. In fact, in any of the aforementioned works, only non-local spatial BC have been considered, while we will focus our attention on boundary conditions which are non-local in time. We refer to \cite{Inf,Inf1,Inf2, Web}, and references therein, for concrete applications that can be potentially studied in the light of the approach that we develop in our work.

In particular, our study exploits the theory of delay equations, see, e.g., \cite{Bat1,Bat2}, so that we will lift the time-delayed boundary conditions to have values in a suitable infinite dimensional path space, showing that the corresponding differential operator does in fact generate a strongly continuous semigroup on an appropriate space of paths.

The work is structured as follows: in Sec. \ref{SEC:GF} we will introduce the setting and the main notations; in Sec. \ref{SEC:InfGen}, exploiting the theory of delay operators, we will introduce the infinite dimensional product space we will work in, also showing that we can rewrite our equation as an infinite dimensional problem where the differential operator generates a strongly continuous semigroup, this immediately lead to the wellposedeness of the abstract Cauchy problem; in Sec. \ref{SEC:AbC} we will introduce a stochastic multiplicative perturbation of Brownian type, showing the existence and uniqueness of a \textit{mild solution}, in a suitable sense, under rather mild assumptions on the coefficients; finally, in Sec. \ref{SEC:Opt}, we provide an application of the developed theory to a stochastic optimal control problem.

\section{General framework}\label{SEC:GF}

Let us consider a finite, connected network identified with a finite graph  composed by $n \in \NN$ vertices $\Ver_1,\dots,\Ver_n$, and by $m \in \NN$ edges $e_1, \dots, e_m$ which are assumed to be normalized on the interval $[0,1]$. 
%We will consider a finite connected graph $\mathbb{G}$. 
Moreover, we will  assume that on the nodes $\mathrm{v}_1, \dots, \mathrm{v}_n$ of $\GG$ are endowed with dynamic boundary conditions to be specified later on.

We would like to recall that in \cite{Car3,CDP,DPZ}, a diffusion problem has been considered, stated on a finite graph, where the boundary conditions exhibit  non-local behaviour, namely what happens on a given node also depends on the state of the remaining nodes, even without a direct connection. In the present work, we will consider a different type of non-local condition, studying a diffusion on a finite graph where the boundary conditions, at a given time, are affected  by the present value of the state equation on each nodes, as well as by the  past values of the underlying dynamic.

In particular we exploiting the semigroup theory, see, e.g. \cite{Eng} for a detailed  introduction to semigroup theory  and \cite{MugB} to what concerns its application on networks, to show how to rephrase our main problem  as an abstract Cauchy problem, so that the well posedness of the solution will be linked to the fact that a certain matrix operator generates a $C_0-$semigroup on a suitable, infinite dimensional, space.

In what follows we will employ the following notation: we will use the Latin letter $i,j,k=1,\dots,m$, $m \in \mathbb{N}^+$,  to denote the edges, hence $u_i$ it will be a function on the edge $e_i$, $i=1,\dots, m$; while we will use Greek letters $\alpha,\beta, \gamma = 1,\dots,n$, $n \in \mathbb{N}^+$,   to denote the vertexes,  consequently $d_\alpha$ it will be a function evaluated at the node $v_\alpha$, $\alpha = 1,\dots,n$. 

To describe the  graph structure we use the so-called \textit{incidence matrix} $\Phi = \left (\phi_{\alpha,i}\right )_{(n+1)\times m}$, defined as $\Phi:= \Phi^+-\Phi^-$, where $\Phi^+ = \left (\phi_{\alpha,i}^+\right )_{(n+1)\times m}$,  resp. $\Phi^- = \left (\phi_{\alpha,i}^-\right )_{(n+1)\times m}$, is the \textit{incoming incidence matrix}, resp. the \textit{outgoing incidence matrix}. Let us note that $\phi^+_{\alpha,i}$, resp. $\phi^-_{\alpha,i}$, takes value $1$ whenever the vertex $v_\alpha$ is the initial point, resp. the terminal point, of the edge $e_i$, and $0$ otherwise, that is it holds
\[
\phi^+_{\alpha,i} = 
\begin{cases}
1 & \Ver_\alpha = e_i(0)\, ,\\
0 & \mbox{ otherwise } 
\end{cases} \, ,\quad 
\phi^-_{\alpha,i} = 
\begin{cases}
1 & \Ver_\alpha = e_i(1)\, ,\\
0 & \mbox{ otherwise }  \, 
\end{cases}
\;,
\]

moreover, if $|\phi_{\alpha,i}|=1$,   the edge $e_i$ is called \textit{incident} to the vertex $\Ver_\alpha$ and  accordingly, we define 
\[
\Gamma(\Ver_\alpha) = \{i \in \{1,\dots,m\} \, : \, |\phi_{\alpha i}| =1\}\, ,
\]
as the set of incident edges to the vertex $\Ver_\alpha$.

Taking into consideration the above introduced notations, we state the following diffusion problem on the finite and connected graph $\GG$

{\footnotesize
\begin{equation}\label{EQN:Sys}
\begin{cases}
\dot{u}_j (t,x) =  \left (c_{j} u'_j\right )' (t,x)  \, , \quad t \geq 0\, ,\, x \in (0,1)\, , \, j = 1,\dots,m\, ,\\[1.5ex]
u_j(t,\mathrm{v}_\alpha) = u_l(t,\mathrm{v}_\alpha) =: d^\alpha(t)\, , \quad t \geq 0\, ,\, l,\, j \in \Gamma(\mathrm{v}_\alpha)\, , \, j = 1,\dots,m\, ,\\[1.5ex]
\dot{d}^\alpha (t) = - \sum_{j=1}^m \phi_{j \alpha} u'_j(t,\mathrm{v}_\alpha) + b_{\alpha} d^\alpha(t) + \int_{-r}^0 d^\alpha (t+\theta) \mu(d \theta)\, ,\quad t \geq 0 \, ,\, \alpha=1,\dots,n\, ,\\[1.5ex]
u_j(0,x) = u^0_j(x) \, ,\quad x \in (0,1)\, ,\, j=1,\dots,m\, ,\\[1.5ex]
d^\alpha(0) = d_\alpha^0\, ,\quad \alpha=1,\dots,n\, ,\\[1.5ex]
d^\alpha(\theta) = \eta_\alpha^0(\theta) \, ,\quad \theta \in [-r,0] \, ,\, \alpha=1,\dots,n\, .\\
\end{cases}
\end{equation}
}
where $\mu \in \mathcal{M}([-r,0])$ and $\mathcal{M}([-r,0])$ is the set of Borel measure on $[-r,0]$, being $r>0$ a finite constant.
Before state the main assumptions concerning the terms appearing in \eqref{EQN:Sys}, let us make the following
\begin{Remark}
We would like to underline that the approach we are going to develop can be generalized, exploiting the  same techniques, to the case where only $0 < n_0 < n$ nodes have dynamics conditions, whereas the remaining $n-n_0$ nodes exhibit standard  \textit{Kirchhoff type} conditions. Since our  interest mainly concerns the study of dynamic boundary conditions, and to consider a mixed boundary type conditions does not affect neither the approach nor the final result, for the sake of simplicity we will assume that all the $n$ nodes composing the graph are endowed with  dynamic boundary conditions.\end{Remark}

With respect to the definition of the terms we have introduced in \eqref{EQN:Sys}, in order to consider the  diffusion problem on $\GG$ , we assume the following to hold

\begin{Assumptions}\label{ASS:1}
\begin{description}
\item[(i)] for any $j=1,\dots,m,$ the function $c_j \in C^1([0,1])$, while $c(x)>0$ for a.a. $x \in [0,1]$;\\
\item[(ii)] for any $\alpha=1,\dots,n,$ we have that $b_\alpha \leq 0$, moreover there exists at least one $\alpha \in \{ 1,\dots,n\},$ such that  $b_\alpha<0$.
\end{description}
\end{Assumptions}

The typical approach concerning the study of delay differential equations consists in lifting the underlying process, which originally takes values in a finite dimensional space, to a suitable infinite dimensional path space, usually the space of square integrable Lebesgue functions or the space of continuous functions. 

In particular, we consider the following Hilbert spaces 
\[
\begin{split}
&X^2 := \left (L^2([0,1])\right )^m\, ,\quad Z^2 := L^2([-r,0];\RR^n)\, ,\\
& \XX := X^2 \times \RR^n \, , \quad   \EP := \XX \times Z^2\, ,
\end{split}
\]
equipped with the standard graph norms and scalar products. Since we are interested in applying the aforementioned {\it lifting} procedure to rewrite the dynamic of  the $\RR^n-$valued process $d$ as it takes values  in an infinite dimensional space, we introduce the notion of \textit{segment}. In particular, we consider the process $d :[-r,T] \to \RR^n$, and, for any $t \geq 0$,  we define the \textit{segment} as
\begin{equation}\label{segment}
d_t :[-r,0] \to \RR^n\, ,\quad  [-r,0] \ni \theta \mapsto d_t(\theta) := d(t + \theta) \in \RR^n \, .
\end{equation}
As it is standard in dealing with 
delay equation,  we denote by $d(t)$ the present $\RR^n-$value of the process $d$, whereas  $d_t$ stands for the \textit{segment} of the process $d$, i.e. $d_t = \left (d(t+\theta)\right )_{\theta \in [-r,0]}$. More precisely, we have
\[
\begin{split}
u(t) &:= \left (u_1(t),\dots, u_m(t)\right )^T \in X^2\, ,\\
d(t) &:= \left (d^1(t),\dots, d^n(t)\right )^T \in \RR^n\, , \\
d_t &:= \left (d^1_t,\dots, d^n_t\right )^T \in Z^2 \, .
\end{split}
\]
Exploiting latter notations, we can  rewrite the system \eqref{EQN:Sys}, as follows
\begin{equation}\label{EQN:InfDelayGenerale}
\begin{cases}
\dot u(t) = A_m u(t)\, , \quad t \in [0,T]\, ,\\
\dot d(t) = C u(t) + \Phi d_t + B d(t)\, , \quad t \in [0,T]\, ,\\
\dot d_t = A_{\theta} d_t\, , \quad t \in [0,T]\, ,\\
L u(t) = d(t)\, ,\\
u(0) = u_0 \in X^2 \, ,\quad d_0 = \eta \in Z^2 \, ,\quad  d(0) = d^0 \in \RR^n \, ,
\end{cases}
\end{equation}

%
%being $A_m$, $A_\theta$, $C$, $\Phi$ and $\tilde{B}$ suitable operators to be introduced later, or in a more compact fashion
%\begin{equation}
%\begin{cases}
%\dot \uu(t) = \AAA \uu(t)\, , \quad t \in [0,T]\, ,\\
%\uu(0) = \uu_0 \in \EP \, ,
%\end{cases}
%\end{equation}
%with
%\[
%\AAA = \left( \begin{array}{ccc}
%A_m & 0 & 0 \\
%C & B & \Phi \\
%0 & 0 & A_\theta\\
%\end{array} \right)\, , \quad \uu(t) = \left (u(t),d(t),d_t\right )^T\, ,
%\]
%with $\AAA$ having a suitable domain $D(\AAA)$ to be introduced in a while.
where $A_m$ is the differential operator defined by
\[
A_m u (t,x) = \left( \begin{array}{ccc}
\frac{\partial}{\partial x}\left (c_j(x) \frac{\partial}{\partial x} u_1(t,x)\right ) & 0 & 0 \\
0 & \ddots & 0  \\
0 & 0 & \frac{\partial}{\partial x}\left (c_m(x) \frac{\partial}{\partial x} u_m(t,x) \right )\\
\end{array} \right)\, ,
\]
and such that  $A_m : D(A_m) \subset X^2 \to X^2$, with domain
\[
D(A) := \left \{ u \in \left (H^2([0,1])\right )^m \, : \, \exists d \in \RR^n \, : \, L u = d \right \}\, ,
\]
where $L: \left (H^1([0,1])\right )^m \to \RR^n$ is the following \textit{boundary evaluation operator}
\[
L u(t,x) := \left (d^1(t), \dots, d^n(t)\right )^T\, ,\quad d^\alpha(t) := u_j(t,\mathrm{v}_\alpha) \,, \quad  j \in \Gamma(\mathrm{v}_\alpha) \, .
\]
We underline that the  operator $(A,D(A))$ just defined, generates a $C_0-$semigroup on the space $X^2$, see, e.g., \cite{BZM,DPZ,Mug}.
Moreover, in writing system \eqref{EQN:InfDelayGenerale}, we also made use of the so-called \textit{feedback operator} $C : D(A) \to \RR^n$, which is defined as follows
\[
C u(t,x) := \left (-\sum_{j=1}^m \phi_{j 1} u'_j(t,\mathrm{v}_1) , \dots, -\sum_{j=1}^m \phi_{j n} u'_j(t,\mathrm{v}_n) \right )^T \, ,
\]
furthermore, we have set $B$ to be the following $n \times n$ diagonal matrix
\[
B = \left( \begin{array}{ccc}
b_1 & 0 & 0 \\
0 & \ddots & 0  \\
0 & 0 & b_n\\
\end{array} \right)\, ,
\]
where  $b_\alpha$, $\alpha=1,\dots,n$, satisfy assumptions \ref{ASS:1}; also the operator 
\begin{equation}\label{delayoperator}
\Phi:C([-r,0];\RR^n) \to \RR^n\, ,
\end{equation}
defined by
\begin{equation}\label{EQN:BDel}
\Phi(\eta) = \int_{-r}^0 \eta(\theta) \mu(d \theta)\, ,
\end{equation}
where $\mu$ is a measure of bounded variation. Notice that a particular case of the present situation is the discrete delay case, that is $\mu = \delta_{x_0}$, being $ \delta_{x_0}$ the Dirac measure centred at $x_0 \in [-r,0)$. Eventually, we have denoted by $A_\theta: D(A_\theta) \subset Z^2 \to Z^2$, the linear differential operator defined by
\[
A_\theta \eta := \frac{\partial}{\partial \theta} \eta(\theta)\, ,\quad D(A_\theta) = \{ \eta \in H^1([-r,0];\RR^n) \, : \, \eta(0) = d^0\}\, ,
\]
 where the derivative $\frac{\partial}{\partial \theta}$ has to be intended as the weak distributional derivative in $Z^2$.

\begin{Remark}
A particular case of the setting introduced above is given by choosing  the so-called \textit{continuous delay operator} $\Phi d_t = \int_{-r}^0 d^u(t+\theta)\mu(d \theta)$, which ensures that  \eqref{EQN:Sys} satisfies the aforementioned assumptions. Another possible  choice is represented by the \textit{discrete delay operator} $\Phi d_t = d^u(t-r)$, which is obtained by the previous one taking $\mu = \delta_{-r}$, where  $\delta_{-r}$ is the Dirac delta centered at $-r$. In what follows we do not specify the particular form of the \textit{delay operator}, in order to prove our results   in the general case of a bounded linear operator $\Phi$.
\end{Remark}

Summing up the previously introduced notation, we can rewrite  equation \eqref{EQN:InfDelayGenerale} more compactly, namely 
\begin{equation}\label{EQN:InfDelayG}
\begin{cases}
\dot \uu(t) = \AAA \uu(t)\, , \quad t \in [0,T]\, ,\\
\uu(0) = \uu_0 \in \EP  \, ,
\end{cases}
\end{equation}
where  $\uu(t) := (u(t),d(t),d_t)^T$, $\uu_0 := (u_0,d^0,\eta)\in \EP$, and the operator $\AAA$ is defined as
\begin{equation}\label{EQN:OpA}
\AAA := \left( \begin{array}{ccc}
A_m & 0 & 0 \\
C & B & \Phi \\
0 & 0 & A_\theta\\
\end{array} \right)\; ,
\end{equation}
with domain $D(\AAA) := D(A_m) \times D(A_\theta)$.
We will  show later that the matrix operator $(\AAA,D(\AAA))$  in equation \eqref{EQN:OpA}, generates a $C_0-$semigroup on the Hilbert space $\EP$, which implies  the wellposedness as well as the uniqueness of the solution, in a suitable sense, for the equation \eqref{EQN:InfDelayG}.

\section{On the infinitesimal generator}\label{SEC:InfGen}

The  present section will be mainly dedicated to the study of  the operator defined in equation \eqref{EQN:OpA}, aiming at proving that it  generates a $C_0-$semigroup.
For the sake of completeness, we recall that the operator $\mathcal{A}$ generates a strongly continuous semigroup in the case that no delay on the boundary is taken into account. In fact, according to the notation introduced within section \ref{SEC:GF}, if we consider the operator 
\begin{equation}\label{EQN:OpAa}
A_{\mathfrak{a}} := \left( \begin{array}{cc}
A_m & 0 \\
C & B \\
\end{array} \right)\, ,
\end{equation}
with domain
\begin{equation}\label{EQN:OpDAa}
D(A_{\mathfrak{a}}) := \left \{ \uu = (u,d) \in \XX \, : \, u \in D(A_m) \, , \, u_j(\mathrm{v}_\alpha) = d^\alpha \,\quad  j \in \Gamma(\mathrm{v}_\alpha) \right \}\, ,
\end{equation}
then we have the following result.

\begin{Proposition}\label{PROP:E!Form}
Let  assumptions \ref{ASS:1} hold true, then the operator $\left (A_{\mathfrak{a}},D(A_{\mathfrak{a}})\right )$ 
%defined in equations \eqref{EQN:OpAa}-\eqref{EQN:OpDAa} 
is self-adjoint, dissipative and has compact resolvent. In particular $A_{\mathfrak{a}}$ generates an analytic $C_0-$semigroup of contractions on the Hilbert space $\XX$. Moreover,  the semigroup $\left (T_{\mathfrak{a}}(t)\right )_{t \geq 0}$, generated by $A_{\mathfrak{a}}$, is uniformly exponentially stable.
\end{Proposition}
\begin{proof}
A proof of the claim can be found in \cite[Prop. 2.4]{BZM}, as well as in \cite[Cor 3.4]{MugR}, nevertheless we give a sketch of it to better clarify the type of methods involved.
We consider the sesquilinear form $\mathfrak{a}: V_{\mathfrak{a}}\times V_{\mathfrak{a}} \to \RR$, defined, for any $\uu= (u,d)$, $\vv=(v,h) \in \XX$, by
\begin{equation}\label{EQN:Forma}
\mathfrak{a}(\uu,\vv) = \sum_{j=1}^m \int_0^1 c_j(x) u_j'(x) v_j'(x) dx + \sum_{\alpha=1}^n b_\alpha d^\alpha h^\alpha\,  .
\end{equation}
and with dense domain $V_{\mathfrak{a}} \subset \XX$ defined as follows
\[
\begin{split}
V_{\mathfrak{a}} &:= \left \{\uu = (u,d) \in \XX \, : \, u \in \left (H^1(0,1)\right )^m \, , \right .\\
& \qquad \qquad \left . u_j(\mathrm{v}_\alpha) = d^\alpha \, ,\quad \alpha = 1,\dots,n \, , \, j \in \Gamma(\mathrm{v}_\alpha)\right \}\: .
\end{split}
\]
Exploiting \cite[Lemma 3.2]{MugR}, it can be shown that the form $\mathfrak{a}$  is symmetric, closed, continuous and positive, then, by \cite[Lemma 3.3]{MugR}, it is associated to the operator $\left (A_{\mathfrak{a}},D(A_{\mathfrak{a}})\right )$, 
%in equation \eqref{EQN:OpAa}--\eqref{EQN:OpDAa}. 
and the result follows by using classical results on sesquilinear forms, see, e.g., \cite{Oua}.
\end{proof}

Using the  operator defined in \eqref{EQN:OpAa}--\eqref{EQN:OpDAa}, and  exploiting a well known perturbation result, it is possible to show that the operator $\left (\AAA,D(\AAA)\right )$ generates a $C_0-$semigroup. We will first prove that the diagonal operator defined as
\begin{equation}\label{EQN:OpDiagA}
\AAA_0 := \left( \begin{array}{cc}
A_\mathfrak{a} & 0 \\
0 & A_\theta\\
\end{array} \right)\, , \quad D(\AAA_0)=D(\AAA)\, ,
\end{equation}
generates a $C_0-$semigroup on the Hilbert space $\EP$.

\begin{Theorem}\label{THM:E!Gen}
Let  assumptions \ref{ASS:1} hold true, then the matrix operator $\left (\AAA_0,D(\AAA_0)\right )$, defined in equation \eqref{EQN:OpDiagA}, generates a $C_0-$semigroup given by
\begin{equation}\label{EQN:SGrT0}
\mathcal{T}_0(t) = \left(
\begin{array}{cc|c}
\qquad {\huge\mbox{{$T_\mathfrak{a}(t)$}}} &&0 \\
 && 0\\ \hline
 0 & T_t & T_0 (t)\\[0.5ex]
\end{array}
\right)\, ,
\end{equation}
where  $T_\mathfrak{a}$ is the $C_0-$semigroup generated by $\left (A_{\mathfrak{a}},D(A_{\mathfrak{a}})\right )$, see equations \eqref{EQN:OpAa}-\eqref{EQN:OpDAa}, $T_0(t)$ is the nilpotent left-shift semigroup
\begin{equation}\label{EQN:Nilpotent}
\left (T_0(t) \eta \right  )(\theta) := \begin{cases}
\eta(t+\theta) & t+\theta \leq 0\, ,\\
0 & t+\theta >0\, , 
\end{cases}\, ,\quad \eta \in Z^2\, ,
\end{equation}
and $T_t: \RR^n \to Z^2$ is defined by
\begin{equation}\label{EQN:Tt}
\left (T_t d \right ) (\theta) := \begin{cases}
e^{(t+\theta)B} d & -t < \theta \leq 0\, ,\\
0 & -r \leq \theta \leq -t\, ,
\end{cases}\, ,\quad  d\in \RR^n \, ,
\end{equation}
$e^{(t+\theta)B} $  being the semigroup generated by the finite dimensional $n \times n$ matrix $B$, as follows
\[
e^{tB} := \sum_{i=0}^\infty\frac{\left (t B\right )^i}{i!}\, .
\]
\end{Theorem}
\begin{proof}
From the strong continuity of $T_\mathfrak{a}$ and $T_0 (t)$ and exploiting the equation \eqref{EQN:Tt}, we  have that  the semigroup $\mathcal{T}_0 (t)$, see equation  \eqref{EQN:SGrT0}, is strongly continuous. Hence, we can compute the resolvent for the semigroup \eqref{EQN:SGrT0}, showing that the corresponding generator is given by \eqref{EQN:OpDiagA}.
To what concerns the resolvent  of the operator $\AAA_0$, namely $R(\lambda,\AAA_0)$, we thus have
\[
R(\lambda,\AAA_0)\XXX = \int_0^\infty e^{-\lambda t} \mathcal{T}_0(t)\XXX dt\, ,\quad \lambda \in \mathbb{C}\, ,\quad \XXX \in \EP\, .
\]
Let us  take $\uu :=(u,d) \in D(A_\mathfrak{a})$ and  $\eta \in H^1([-r,0];\RR^n)$, such that the following holds
\begin{align}
\left (\lambda - A_\mathfrak{a}\right ) (u,d)^T = (v,d^v)^T \, ,\quad (v,h)^T \in \XX  \label{EQN:Res1}\, ,\\
\lambda \eta - \eta' = \zeta \, ,\quad \eta(0) = d \, ,\quad \zeta \in Z^2 \, , \, \label{EQN:Res2}\, ,
\end{align} 
then a solution to equation \eqref{EQN:Res2} is given by
\[
\eta(\theta) = e^{\lambda \theta}\left (d + \int_\theta^0 e^{-\lambda t} \zeta(t) dt\right )\, .
\]
Moreover, if we indicate with $A^0_\theta$ the infinitesimal generator of the nilpotent left shift, namely
\[
A^0_\theta \eta = \eta' \, \quad D(A^0_\theta) = \{\eta \in H^1([-r,0];\RR^n) \, : \, \eta(0) = 0 \}\, ,
\]
we have that its resolvent is given by
\[
\left (R(\lambda,A^0_\theta)\zeta\right )(\theta) = e^{\lambda \theta} \int_{\theta}^0 e^{-\lambda t} \zeta(t) dt \, ,
\]
see, e.g., \cite{Eng}, therefore, taking $\mathbf{Y} =(v,h,\zeta)^T $, the resolvent for $\AAA_0$ reads as follows
\[
\begin{split}
R(\lambda,\AAA_0) \mathbf{Y} &= \left ( R(\lambda,A_\mathfrak{a})(v,h) , e^{\lambda \theta}R(\lambda,B) h + R(\lambda,A^0_\theta) \zeta \right )^T =\\
&=\left(
\begin{array}{cc|c}
\qquad \qquad {\huge\mbox{{$R(\lambda,A_\mathfrak{a})$}}} &&0 \\
 && 0\\ \hline
 0 & e^{\lambda \theta} R(\lambda,B) & R(\lambda,A_\theta^0) (t)\\[0.5ex]
\end{array}
\right) \mathbf{Y} \, .
\end{split}
\]
Summing up, the result follows noticing that
\[
\begin{split}
&\int_0^\infty e^{-\lambda t} \left (T_t d \right )(\theta) dt =\int_{-\theta}^\infty e^{-\lambda t} e^{(t + \theta)B} d(t) dt =\\
&= e^{\lambda \theta} \int_0^\infty e^{(t + \theta)B} d(t) dt = e^{\lambda \theta}R(\lambda,B)\, ,
\end{split}
\]
so that, we have
\[
R(\lambda,\AAA_0) = \int_0^\infty e^{-\lambda t} \mathcal{T}_0(t) dt\, ,
\]
which implies that  the semigroup $\left  (\mathcal{T}_0(t)\right )_{t \geq 0}$, defined in equation \eqref{EQN:SGrT0}, is generated by $\left (\AAA_0,D(\AAA_0)\right )$ in \eqref{EQN:OpDiagA}.
\end{proof}

In what follows we prove that the matrix operator $\left (\AAA,D(\AAA)\right )$ \eqref{EQN:OpA} generates a $C_0-$semigroup on the Hilbert space $\EP$, exploiting a perturbation approach. In particular,  we  exploit firstly the \textit{Miyadera-Voigt perturbation theorem}, see, e.g., \cite[Cor. III.3.16]{Eng}, 
%to show that the operator $\left (\AAA,D(\AAA)\right )$ generates a strongly continuous semigroup.
which states the following
\begin{Theorem}\label{THM:MV}
Let $(G,D(G))$ be the generator of a strongly continuous semigroup $\left (S(t)\right )_{t \geq 0}$, defined on a Banach space $X$, and let $K \in \mathcal{L}\left (\left (D(G),\| \cdot\|_G\right ); X\right )$. Assume that there exist constants $t_0 >0$ and $0 \leq q < 1$, such that
\begin{equation}\label{EQN:MV}
\int_0^{t_0} \| KS(t) x \| dt \leq q \| x\| \, ,\quad \forall \, x \in D(G)\, .
\end{equation}
Then $(G+K,D(G))$ generates a strongly continuous semigroup $\left (U(t)\right )_{t \geq 0}$ on $X$, which satisfies
\begin{equation}\label{EQN:ConvDel}
U(t) x = S(t) x + \int_0^t  S(t-s) K U(s) x ds\, ,
\end{equation}
and 
\[
\int_0^{t_0} \| KU(t)x\| dt \leq \frac{q}{1-q} \|x\|\, ,\quad \forall \, x \in D(G)\, , \, t \geq 0\, .
\]
\end{Theorem}

Let us now to consider the operator matrix
\[
\AAA_1  :=\left( \begin{array}{ccc}
0 & 0 & 0 \\
0 & 0 & \Phi\\
0 & 0 & 0\\
\end{array} \right)\, \in \mathcal{L}\left (D(\AAA_0),\EP\right )\, ,
\]
where $\Phi$ is the delay operator defined in equation \eqref{delayoperator}. Exploiting Theorem \ref{THM:MV} we show that, under a suitable assumption on  $\Phi$, the matrix operator $\AAA= \AAA_0 + \AAA_1$ generates a $C_0-$semigroup on $\EP$. 

\begin{Theorem}\label{THM:GenAGen}
Let  assumptions \ref{ASS:1} hold true, then the operator $\left (\AAA,D(\AAA)\right )$ defined in equation \eqref{EQN:OpA}, generates a strongly continuous semigroup.
\end{Theorem}
\begin{proof}
The result  follows applying  the \textit{Miyadera-Voigt perturbation theorem} \ref{THM:MV}, together with the assumption for the \textit{delay operator} $\Phi$ to be bounded, see equation \eqref{delayoperator}, therefore the perturbation operator $\AAA_1$ is bounded. 
In fact, from the boundness of $\Phi$, we have that, for $\XXX=(u,d,\eta)^T$, it holds
\[
\int_0^{t_0} \left | \AAA_1 \mathcal{T}_0(t) \XXX \right | dt = \int_0^{t_0} \left |\Phi\left ( T_t d + T_0(t) \eta\right ) \right | dt \, .
\]

Thus, following \cite[Example 3.1 (b)]{Bat2} we have that, denoting in what follow by $|\mu|$ the positive Borel measure defined by the total variation of the measure $\mu$,
\[
\begin{split}
& \int_0^{t} \left |\Phi\left ( T_s d + T_0(s) \eta\right ) \right | ds = \int_0^t \left | \int_{-r}^{-s} \eta(s+\theta) \mu(d\theta) + \int_{-s}^0 \left (e^{(s+\theta)B }d\right )\, \mu(d\theta) \right | ds \leq\\
&\leq \int_0^t \int_{-r}^{-s} \left | \eta(s+\theta)\right | |\mu|(d\theta)ds + \int_0^t \int_{-s}^0 \left |e^{(s+\theta)B} d\right | |\mu|(d\theta) ds \leq\\
&\leq \int_{-t}^0 \int_{\theta}^{0} \left | \eta(s)\right |ds |\mu|(d\theta) + \int_{-r}^{-t} \int_{\theta}^{t+\theta} \left | \eta(s)\right |ds |\mu|(d\theta) + \int_0^t \sup_{s \in [0,r] } \left |e^{sB}\right | |d| |\mu| ds \, .
\end{split}
\]

Denoting now by
\[
K := \sup_{s \in [0,r] } \left |e^{sB} \right |\, ,
\]
we have
\[
\begin{split}
&\int_{-t}^0 \int_{\theta}^{0} \left | \eta(s)\right |ds |\mu|(d\theta) + \int_{-r}^{-t} \int_{\theta}^{t+\theta} \left | \eta(s)\right |ds |\mu|(d\theta) + \int_0^t K |d| |\mu| ds \leq \\
&\leq \int_{-t}^0 \sqrt{-\theta} \left \| \eta\right \|_2 |\mu|(d\theta) + \int_{-r}^{-t} \sqrt{t} \left \| \eta\right \|_2|\mu|(d\theta) + t K |d| |\mu| \leq \\
&\leq  \int_{-r}^{0} \sqrt{t} \left \| \eta\right \|_2|\mu|(d\theta) + t K |d| |\mu|  = \left (\sqrt{t}\|\eta\|_2 + t K |d|\right )|\mu|\, .
\end{split}
\]

Choosing thus $t_0$ small enough such that
\[
q:= \sqrt{t_0} K |\mu| < 1\, ,
\]
we have that
\[
\int_0^{t_0} \|\Phi\left ( T_t d + T_0(t) \eta\right ) \| dt \leq q \|(\eta,d)\|\, ,
\]
choosing thus $t_0$ such that equation \eqref{EQN:MV} is satisfied, the claim therefore follows.
\end{proof}

We note that Theorem\ref{THM:GenAGen} holds for  more general type of \textit{delay operators}, namely taking into consideration weaker assumptions on its definition. In fact, by the result contained in \cite[Th. 1.17]{Bat1}, we have that  $\left (\AAA,D(\AAA)\right )$,  defined in equation \eqref{EQN:OpA}, generates a strongly continuous semigroup  for a general operator
\[
\Phi: H^1([-r,0];\RR^n) \to \RR^n\, ,
\]
provided that  there exist $t_0>0$ and $0 < q <1$ such that
\[
\int_0^{t_0} \| \Phi (S_t \uu + T_0(t) \eta)\|dt \leq q \| (\uu,\eta)\|\, .
\]

\begin{Remark}
The \textit{Miyadera-Voigt perturbation theorem} \ref{THM:MV} implies that the perturbed semigroup $\left (\TT\right )_{t \ge 0}$ is given in terms of the \textit{Dyson--Phillips series}
\begin{equation}\label{EQN:DPS}
\TT x = \sum_{n=0}^\infty \mathcal{T}^n(t) x \, ,
\end{equation}
where each operator $\mathcal{T}^n(t) x$ is defined inductively as
\[
\mathcal{T}^0(t) x:= \mathcal{T}_0(t)x\,,
\]
and
\begin{equation}\label{EQN:DPN}
\mathcal{T}^n(t) x:= \int_0^t \mathcal{T}^{n-1}(t-s) \mathcal{A}_1 \mathcal{T}_0(s) x ds\, .
\end{equation}
\end{Remark}

\section{The perturbed stochastic problem}\label{SEC:AbC}

In the present section we study the system defined  in \eqref{EQN:Sys}  perturbed by a multiplicative Gaussian noise. 
We will carry out our analysis with respect to  the following standard, complete and filtered probability space $\left (\Omega, \mathcal{F},\left (\mathcal{F}_t\right )_{t \geq 0}, \mathbb{P}\right )$, then we define the following system

{\footnotesize
\begin{equation}\label{EQN:SysPerW}
\begin{cases}
\dot{u}_j (t,x) =  \left (c_{j} u'_j\right )' (t,x)+ g_j(t,x,u_j(t,x)) \dot{W}_j^1(t,x)  \, , \\
\qquad \qquad \qquad \qquad \qquad \qquad \qquad \qquad \qquad \qquad t \geq 0\, ,\, x \in (0,1)\, , \, j = 1,\dots,m\, ,\\[1.5ex]
u_j(t,\mathrm{v}_\alpha) = u_l(t,\mathrm{v}_\alpha) =: d^\alpha(t)\, , \quad t \geq 0\, ,\, l,\, j \in \Gamma(\mathrm{v}_i)\, , \, j = 1,\dots,m\, ,\\[1.5ex]
\dot{d}^\alpha (t) = - \sum_{j=1}^m \phi_{j \alpha} u'_j(t,\mathrm{v}_\alpha) + b_{\alpha} d^\alpha(t) + \int_{-r}^0 d^\alpha(t+\theta) \mu(d \theta)+ \tilde{g}_\alpha(t,d^\alpha(t),d^\alpha_t) \dot{W}^2_\alpha (t,\mathrm{v}_{\alpha})\, ,\\
\qquad \qquad \qquad \qquad \qquad \qquad \qquad \qquad \qquad \qquad \qquad \qquad \qquad \qquad  t \geq 0 \, ,\, \alpha=1,\dots,n\, ,\\[1.5ex]
u_j(0,x) = u^0_j(x) \, ,\quad x \in (0,1)\, ,\, j=1,\dots,m\, ,\\[1.5ex]
d^\alpha(0) = d_\alpha^0\, ,\quad \alpha=1,\dots,n\, ,\\[1.5ex]
d^\alpha(\theta) = \eta_\alpha^0(\theta)\, ,\quad \theta \in [-r,0] \, ,\, \alpha=1,\dots,n\, .\\
\end{cases}
\end{equation}
}
where  $W^1_j$ and $W^2_\alpha$, $j=1,\dots,m$, $\alpha=1,\dots, n_0$, are independent $\mathcal{F}_t -$adapted space time Wiener processes to be specified in a while, and $\dot{W}$ indicates the {\it formal} time derivative. In particular $W^1_j$, $j=1,\dots,m$, is a space time Wiener process taking values in $L^2(0,1)$, consequently we  denote by $W^1=(W^1_1,\dots,W^1_m)$ a space time Wiener process with values in $X^2 := \left (L^2(0,1\right )^m$. Similarly, we have that each $W^2_\alpha$, $\alpha=1,\dots,n$, is a space time Wiener process with values in $\RR$, so that we denote by $W^2=(W^2_1,\dots,W^2_n)$ the standard Wiener process with values in $\RR^n$. Eventually, we indicate by $W := (W^1,W^2)$ a standard space time Wiener process with values in $\XX := X^2 \times \RR^n$.

In what follows we require both the assumptions stated in  \ref{ASS:1}, as well as the following
\begin{Assumptions}\label{ASS:3}
\begin{description}
\item[(i)] The functions 
\[
g_j: [0,T] \times [0,1] \times \RR \to \RR\, , \quad j=1,\dots,m\,,
\]
are measurable, bounded and uniformly Lipschitz with respect to  the third component, namely there exist $C_j>0$ and $K_j>0$, such that, for any $(t,x,y_1) \in [0,T] \times [0,1] \times \RR$ and $(t,x,y_2) \in [0,T] \times [0,1] \times \RR$, it holds
\[
|g_j(t,x,y_1)| \leq C_j \, , \quad |g_j(t,x,y_1)-g_j(t,x,y_2)| \leq K_j |y_1 - y_2| \, ;
\]
\item[(ii)] The functions 
\[
\tilde{g}_\alpha : [0,T] \times \RR \times Z^2 \to \RR\, ,\quad \alpha=1,\dots,n_0\,,
\]
are measurable, bounded and uniformly Lipschitz with respect to the second component, namely there exist $C_\alpha>0$ and $K_\alpha>0$, such that, for any $(t,u,\eta) \in [0,T] \times \RR \times Z^2$ and $(t,v,\zeta) \in [0,T] \times \RR \times Z^2$, it holds
\[
|\tilde{g}_\alpha(t,u,\eta)| \leq C_\alpha \, , \quad |\tilde{g}_\alpha(t,u,\eta)-\tilde{g}_\alpha(t,v,\zeta)| \leq K_\alpha( |u-v|_n + | \eta - \zeta |_{Z^2})\, .
\]
\end{description}
\end{Assumptions}
Using previously introduced notations, the problem in \eqref{EQN:SysPerW} can be rewritten as the following  abstract infinite dimensional Cauchy problem 

\begin{equation}\label{EQN:Mild}
\begin{cases}
d \XXX(t) = \mathcal{A}\XXX(t) dt + G(t,\XXX(t)) dW(t) \, ,\quad t \geq 0\, ,\\
\uu(0) = \uu_0 \in \EP\, ,
\end{cases}
\end{equation}
where $\AAA$ is the operator introduced in \eqref{EQN:OpA},  the map $G$ is defined as the following application
\[
G:[0,T] \times \EP \to \mathcal{L}(\XX;\EP)\,,
\]
being $\mathcal{L}(\XX;\EP)$ the space of linear and bounded operator from $\XX$ to $\EP$, equipped with standard norm $|\cdot|_{\mathcal{L}}$,  other terms are intended such as they have been  defined within Sec. \ref{SEC:InfGen}, and  $W = (W^1,W^2)$ is a $\XX-$valued standard Brownian motion.

In particular, if $\XXX = (\uu,\eta)^T = (u,y,\eta) \in \EP$, and $ \vv = (v,z) \in \XX$, then $G$ is defined as
\begin{equation}\label{EQN:G}
G(t,\XXX)\vv = \left (\sigma_1(t,u)v,\sigma_2(t,y,\eta)z,0\right )^T \, ,
\end{equation}
with
\[
\begin{split}
&\left (\sigma_1(t,u)v\right )(x) = \left (g_1(t,x,u_1(t,x)),\dots,g_m(t,x,u_m(t,x)) \right )^T\, ,\\
& \sigma_2(t,y,\eta)z = \left (\tilde{g}_1(t,y_1,\eta)z_1,\dots,\tilde{g}_{n}(t,y_{n},\eta)z_{n}\right )^T\, .
\end{split}
\]
Our next step concerns how to obtain a \textit{mild solution} to equation \eqref{EQN:Mild}, namely a solution  defined in the following sense
\begin{Definition}\label{DEF:Mild}
We will say that $\XXX$ is \textit{mild solution} to equation \eqref{EQN:Mild} if it is a mean square continuous $\EP-$valued process, adapted to the filtration generated by $W$, 
 such that, for any $t \geq 0$, we have that $\XXX \in L^2\left (\Omega,C([0,T];\EP)\right )$ and it holds
\begin{equation}\label{EQN:MildSolApp}
\XXX(t) = \TT \XXX_0 +  \int_0^t \mathcal{T}(t-s) G(s,\XXX(s)) dW(s) \,, \quad t \geq 0\, .
\end{equation}
%The last term in equation \eqref{EQN:MildSolApp} is called  \textit{stochastic convolution}, often denoted by
%	\[
%	W^{A}_t := \int_0^t e^{(t-s)A} \sigma dW_s\, .
%	\]
\end{Definition}
In general, in order to guarantee the existence and uniqueness of a mild solution to equation \eqref{EQN:Mild}, we have to require that
%the stronger property of 
\[
G:[0,T]\times \EP \to \mathcal{L}_2(\XX;\EP)\, ,
\]
being $\mathcal{L}_2(\XX;\EP)$ the space of \textit{Hilbert-Schmidt} operator from $\XX$ to $\EP$ equipped with its standard norm denoted as $|\cdot|_{HS}$, see, e.g., \cite[Appendix C]{Dap}. Nevertheless ,  when dealing with  a diffusion problem where the leading term is a second order differential operator, it is enough to require that $G$ takes value in $\mathcal{L}(\XX;\EP)$ since, in this particular case, the map $G$ inherits the needed regularity from the analytic semigroup generated by the second order differential operator. On the other hand, if we consider  a delay operator then,  due to the presence of the first order differential operator $A_\theta$, the operator $\AAA$, defined in equation \eqref{EQN:OpA}, does not generate an analytic semigroup on the space $\EP$. The latter suggests that it seems reasonable  to require $G$ to take values in  $\mathcal{L}_2(\XX;\EP)$, in order to have both existence and uniqueness for a solution to equation \eqref{EQN:Mild}.
In what follows, we will show that, since $A_\mathfrak{a}$ generates an analytic semigroup, and exploiting the particular form for $G$ in equation \eqref{EQN:G}, we have that $\mathcal{T}(t)G(s,\XXX)$ belongs to $\mathcal{L}_2(\XX;\EP)$, hence, by assumptions \ref{ASS:3} on the functions $g$ and $\tilde{g}$, the existence and uniqueness of a mild solution to equation \eqref{EQN:Mild} follows. 

The next result will be  later used in order to show the existence and uniqueness of a mild solution to equation \eqref{EQN:Mild}. 

\begin{Proposition}\label{PRO:HSG}
Let  assumptions \ref{ASS:1}--\ref{ASS:3} hold true, then the map $G:[0,T] \times \EP \to \mathcal{L}(\XX,\EP)$, defined in equation \eqref{EQN:G}, satisfies:
\begin{description}
\item[(i)] for any $\uu \in \XX$ the map $G(\cdot,\cdot) \uu : [0,T] \times \EP \to \EP$, is measurable;\\
\item[(ii)] for any $T >0$, there exists a constant $M>0$, such that for any $t \in [0,T]$ and $s \in [0,T]$, and for any $\XXX$, $\YYY \in \EP$, it holds
\begin{align}
&|\TT G(s,\XXX)|_{HS} \leq M t^{-\frac{1}{4}}(1+|\XXX|_{\EP}) \label{EQN:E!1}\, ,\\
&|\TT G(s,\XXX)-\TT G(s,\YYY)|_{HS} \leq M t^{-\frac{1}{4}}|\XXX-\YYY|_{\EP} \label{EQN:E!2}\, ,\\
&|G(s,\XXX)|_{\mathcal{L}} \leq M (1+|\XXX|_{\EP}) \label{EQN:E!3}\, .
\end{align}
\end{description}
\end{Proposition}
\begin{proof}
Point $\textbf{(i)}$  and  \eqref{EQN:E!3}  in point $\textbf{(ii)}$, immediately follow from assumptions \ref{ASS:3}.

%To prove the inequality \eqref{EQN:E!1} in point $\textbf{(ii)}$, we proceed analogously as in the proof of Prop. \ref{PRO:HSXE}. In particular, let us first consider the unperturbed semigroup $\mathcal{T}_0$ and let
% $\{\phi_k\}_{k \in \NN}$, resp. $\{e_k\}_{k=1}^n $, resp. $\{\tilde{\phi}_k\}_{k \in \NN}$, resp.  $\{\psi_k\}_{k \in \NN}$, resp. $\{\tilde{\psi}_k\}_{k \in \NN}$, be an orthonormal basis in $X^2$, resp. in $\RR^n$, resp.  a basis in $\XX$, resp. an orthonormal basis in $Z^2$ , resp. a basis in $\EP$, and, for the sake of clarity, let us denote with $M>0$ different constants throughout  what follows. 
 Let $\{\tilde{\phi}_i\}_{i=1}^\infty$, resp. $\{\phi_i\}_{i=1}^\infty$, resp. $\{e_i\}_{i=1}^n$, resp. $\{\psi_i\}_{i=1}^\infty$, be an orthonormal basis in $\XX$, resp. in $X^2$, resp. in $\RR^n$, resp. in $Z^2$.

Let us thus first consider the unperturbed semigroup $\mathcal{T}_0$ given in equation \eqref{EQN:SGrT0}, and let us show that
\[
|\mathcal{T}_0(t) G(s,\XXX) |_{HS} \leq M t^{-\frac{1}{4}}(1+|\XXX|_{\EP}) \, ,
\]
for a suitable constant $M$.

Exploiting  the explicit form for $G$, see equation \eqref{EQN:G}, we have that 
\begin{equation}\label{EQN:HSSomma}
\begin{split}
| \mathcal{T}_0(t) G(s,\XXX) |_{HS}^2 &= \sum_{j,k \in \NN} \left \langle T_\mathfrak{a}(t)(\sigma_1(s,u),\sigma_2(s,d^u,\eta)) \tilde{\phi}_j,\tilde{\phi}_k\right \rangle_{\XX}  +\\
&+ \sum_{i=1}^n  \sum_{k \in \NN} \langle T_t \sigma_2(s,d^u,\eta) e_j,\psi_k\rangle_{Z^2}\, .
\end{split}
\end{equation}

Since $T_\mathfrak{a}$ is self-adjoint and by \cite[Prop. 10]{BCM}, we have that 
\begin{equation}\label{EQN:StimaHSTa}
\begin{split}
&\sum_{j,k \in \NN} \left \langle T_\mathfrak{a}(t)(\sigma_1(s,u),\sigma_2(s,d^u,\eta)) \tilde{\phi}_j,\tilde{\phi}_k\right \rangle_{\XX} =\\
&=\sum_{j,k \in \NN} \left \langle(\sigma_1(s,u),\sigma_2(s,d^u)) \tilde{\phi}_j,T_\mathfrak{a}(t)  \tilde{\phi}_k \right \rangle_{\XX} \leq\\
&\leq \|(\sigma_1(s,u),\sigma_2(s,d^u))\|_{\mathcal{L}(\XX)} |T_\mathfrak{a}(t)|_{\mathcal{L}_2(\XX)} \leq |G(s,\XXX)|_{\mathcal{L}(\XX;\EP)} |T_\mathfrak{a}(t)|_{\mathcal{L}_2(\XX)} \leq\\
&\leq  M t^{-\frac{1}{2}}(1+ |\XXX|_{\EP})\, .
\end{split}
\end{equation}

Concerning the second term in the right hand side of equation \eqref{EQN:HSSomma}, we have that the following holds for any $e_i$
\begin{equation}\label{EQN:HSDiag}
\left (T_t e_i\right ) = \begin{cases}
(0,\dots,0,e^{(t+\theta)b_i},0,\dots,0) &\, ,\, -t < \theta <0\, ,\\
0 &  \, ,\, -r \leq \theta \leq -t\, ,
\end{cases}
\end{equation}
hence, by assumptions \ref{ASS:3}, we also obtain
\[
\langle T_t \sigma_2(s,d^u,\eta) e_i,\psi_k \rangle_{Z^2} = \int_{-t}^0 e^{(t+\theta)b_i} \sigma_2(s,d^u,\eta)  \psi_k d\theta < \infty\, ,
\]
which implies that the second sum on the right hand side of \eqref{EQN:HSSomma} is finite. Moreover, because $\RR^n$ is finite dimensional and $\mathcal{L}_2(\RR^n;Z^2) =\mathcal{L}(\RR^n;Z^2) $, from equations \eqref{EQN:HSSomma}--\eqref{EQN:StimaHSTa}, we immediately have that
the following holds
\begin{equation}\label{EQN:EstUnpSgr}
|\mathcal{T}_0(t) G(s,\XXX)|_{HS} \leq M t^{-\frac{1}{4}}(1+|\XXX|_{\EP})\, .
\end{equation}

In order to prove the claim for the perturbed semigroup $\left (\mathcal{T}(t)\right )_{t \geq 0}$ let us consider Theorem \ref{THM:MV} so that $\left (\mathcal{T}(t)\right )_{t \geq 0}$ is given by equation \eqref{EQN:ConvDel}; in particular we have
\begin{equation}\label{EQN:GenOper}
\mathcal{T}(t) G(s,\XXX) \vv =\mathcal{T}_0(t) G(s,\XXX) \vv + \int_0^t \mathcal{T}_0(t-s) \Phi \mathcal{T}(s) G(s,\XXX) \vv ds\, .
\end{equation}

Let us denote in what follows for short
\begin{equation}\label{EQN:SOper}
\mathcal{T}(t) G(s,\XXX) =\left(\begin{array}{c} \mathcal{S}^1(t) \\ \mathcal{S}^2(t)\\\mathcal{S}^3(t) \end{array}\right) = \mathcal{S}(t)\, .
\end{equation}

Using therefore the particular form for the delay operator given in equation \eqref{EQN:BDel} together with equations \eqref{EQN:EstUnpSgr}--\eqref{EQN:GenOper}, we obtain for $q >0 $ and $t \geq 0$,
\[
\begin{split}
\left |\mathcal{T}_0(q) \mathcal{S}(t)\right |_{HS} &\leq M(t+q)^{-\frac{1}{4}}(1+|\XXX|_{\EP}) +\\
&+ \left |\int_0^t \mathcal{T}_0(t-s+q) \int_{-r}^0 S(s+\theta)\mu(d\theta) ds\right |_{HS} \leq \\
&\leq M(t+q)^{-\frac{1}{4}}(1+|\XXX|_{\EP}) +\\
&+ |\mu| \sup_{\theta \in [-r,0]}\int_0^t \left |\mathcal{T}_0(t-s+q)S(s+\theta)\right |_{HS} ds\, ,
\end{split}
\]

where $|\mu|$ is the total variation of the measure $\mu$.

Thus, from above equation, for a fixed time $\tilde{T} \in [0,T]$, we have for $t+q \leq \tilde{T}$,
\begin{equation}\label{EQN:IneqStima}
\begin{split}
\sup_{t+q \leq \tilde{T}} q^{\frac{1}{4}} \left |\mathcal{T}_0(q) \mathcal{S}(t)\right |_{HS} &\leq M(1+|\XXX|_{\EP}) +\\
&+ |\mu| \sup_{t+q \leq \tilde{T}} q^\frac{1}{4} \left |\mathcal{T}_0(q) \mathcal{S}(t)\right |_{HS} \int_0^t(t-s)^{-\frac{1}{4}}ds\, .
\end{split}
\end{equation}

As regard $\left |\mathcal{T}_0(q) \mathcal{S}(t)\right |_{HS}$ appearing in the right hand side of equation \eqref{EQN:IneqStima}, denoting for short
\[
\mathcal{T}_0(q) \mathcal{S}(t) = \left(\begin{array}{c} \mathcal{V}^1(q) \\ \mathcal{V}^2(q)\\\mathcal{V}^3(q) \end{array}\right) \, ,
\]
it immediately follows from the computation above that
\[
\left |\mathcal{T}_0(q)\left(\begin{array}{c} \mathcal{S}^1(t) \\ \mathcal{S}^2(t)\\ 0 \end{array}\right) \right |_{HS} <\infty\, ;
\]
noticing thus that from the property of the delay semigroup it holds
\[
\left (\mathcal{V}^3(q)\right ) (\theta) = \left (\mathcal{V}^2(q+\theta) \Ind{ \{q+\theta \geq 0\}}\right )_{\theta \in [-r,0]}\, ,
\]
we immediately have that 
\[
\left |\mathcal{V}^3(q)\right |_{\mathcal{L}_2\left (\XX;Z^2\right )}<\infty\, ,
\]
and we can therefore conclude that
\[
\left |\mathcal{T}_0(q) \mathcal{S}(t)\right |_{HS} = \left |\left(\begin{array}{c} \mathcal{V}^1(q) \\ \mathcal{V}^2(q)\\\mathcal{V}^3(q) \end{array}\right)\right |_{HS} <\infty\, ,
\]
and thus the right hand side in equation \eqref{EQN:IneqStima} is finite.

We can therefore choose $\tilde{T}$ independent of $\XXX$ and $s$, such that the following holds
\begin{equation}\label{EQN:IneqStimaBis}
\sup_{t+q \leq \tilde{T}} q^\frac{1}{4} \left |\mathcal{T}_0(q) \mathcal{S}(t)\right |_{HS} \leq \tilde{M} (1+|\XXX|_{\EP}) \, ,
\end{equation}
with $\tilde{M}$ a suitable constant. Therefore, from equations \eqref{EQN:IneqStima}--\eqref{EQN:IneqStimaBis}, from equation \eqref{EQN:GenOper} we thus have for all $t \in (0,\tilde{T}]$,
\[
\left | \mathcal{T}(t) G(s,\XXX)\right |_{HS} \leq Mt^{-\frac{1}{4}}(1+|\XXX|_{\EP}) + \tilde{M}\left (\int_{0}^t (t-s)^{-\frac{1}{4}}ds\right )(1+|\XXX|_{\EP})\, ;
\]
we thus immediately have that, for all $t \in (0,\tilde{T}]$,
\begin{equation}\label{EQN:IneqStimaFin}
\left | \mathcal{T}(t) G(s,\XXX)\right |_{HS} \leq \bar{M} t^{-\frac{1}{4}}(1+|\XXX|_{\EP})\, ,
\end{equation}
with $\bar{M}$ a given constant. Then, by the semigroup property for $\left (\mathcal{T}(t)\right )_{t \geq 0}$, we can extend estimate \eqref{EQN:IneqStimaFin} for all $t \in [0,T]$.

Finally, the proof of the inequality \eqref{EQN:E!2} in $\textbf{(ii)}$ proceeds the same way as the latter one.
\end{proof}
Summing up previous results, we are now in position to state the following
\begin{Theorem}\label{TH:EUMildSol}
Let  assumptions \ref{ASS:1}--\ref{ASS:3} hold true, then there exists a unique mild solution, in the sense of Definition \ref{DEF:Mild}, to equation \eqref{EQN:Mild}.
\end{Theorem}
\begin{proof}
The result follows by \cite[Th. 5.3.1]{DapE}, see also \cite{DPZ}, together with proposition \ref{PRO:HSG}.
\end{proof}

\subsection{Existence and uniqueness for the non-linear equation}\label{SEC:Lip}

The present subsection is devoted to the generalisation of the existence and uniqueness of a mild solution, see Th. \ref{TH:EUMildSol}, to the abstract formulation, see eq. \eqref{EQN:Mild}, of the problem stated by eq. \eqref{EQN:SysPerW}. In particular we shall consider the addition of a  non-linear Lipschitz perturbation. 
The notation used in what follows is as in previous sections. 

We will thus focus on the following non-linear stochastic dynamic boundary value problem

{\footnotesize
\begin{equation}\label{EQN:SysPerWNL}
\begin{cases}
\dot{u}_j (t,x) =  \left (c_{j} u'_j\right )' (t,x)+f_j(t,x,u_j(t,x))+ g_j(t,x,u_j(t,x)) \dot{W}_j^1(t,x)  \, , \\
\qquad \qquad \qquad \qquad \qquad \qquad \qquad \qquad \qquad \qquad t \geq 0\, ,\, x \in (0,1)\, , \, j = 1,\dots,m\, ,\\[1.5ex]
u_j(t,\mathrm{v}_\alpha) = u_l(t,\mathrm{v}_\alpha) =: d^\alpha(t)\, , \quad t \geq 0\, ,\, l,\, j \in \Gamma(\mathrm{v}_i)\, , \, j = 1,\dots,m\, ,\\[1.5ex]
\dot{d}^\alpha (t) = - \sum_{j=1}^m \phi_{j \alpha} u'_j(t,\mathrm{v}_\alpha) + b_{\alpha} d^\alpha(t) + \int_{-r}^0 d^\alpha(t+\theta) \mu(d \theta)+ \tilde{g}_\alpha(t,d^\alpha(t),d^\alpha_t) \dot{W}^2_\alpha (t,\mathrm{v}_{\alpha})\, ,\\
\qquad \qquad \qquad \qquad \qquad \qquad \qquad \qquad \qquad \qquad \qquad \qquad \qquad \qquad  t \geq 0 \, ,\, \alpha=1,\dots,n\, ,\\[1.5ex]
u_j(0,x) = u^0_j(x) \, ,\quad x \in (0,1)\, ,\, j=1,\dots,m\, ,\\[1.5ex]
d^\alpha(0) = d_\alpha^0\, ,\quad \alpha=1,\dots,n\, ,\\[1.5ex]
d^\alpha(\theta) = \eta_\alpha^0(\theta)\, ,\quad \theta \in [-r,0] \, ,\, \alpha=1,\dots,n\, .\\
\end{cases}
\end{equation}
}

In what follows, besides assumptions \ref{ASS:1} and \ref{ASS:3} and  in order to deal with functions $f_j$ appearing in eq. \eqref{EQN:SysPerWNL}, we also require the following
%In the present section besides assumptions \ref{ASS:1}--\ref{ASS:3}, we will also assume the following
\begin{Assumptions}\label{ASS:4}
The functions 
\[
f_j :[0,T] \times [0,1] \times \RR \to \RR \, ,\quad j=1, \dots, m\, ,
\]
are measurable mappings, bounded and uniformly Lipschitz continuous with respect to the third component, namely, for $j=1, \dots, m$, there exist positive constants $C_j$ and $K_j$, such that, for any $(t,x,y_1) \in [0,T] \times [0,1] \times \RR $ and any $(t,x,y_2) \in [0,T] \times [0,1] \times \RR $, it holds
\[
|f_j(t,x,y_1)| \leq C_j \, , \quad |f_j(t,x,y_1)-f_j(t,x,y_2)| \leq K_j |y_1 -y_2|\: .
\]
\end{Assumptions}
Proceeding similarly to what is seen in Sec. \ref{SEC:AbC}, we reformulate equation \eqref{EQN:SysPerWNL} as an abstract Cauchy problem as follows
\begin{equation}\label{EQN:MildLip}
\begin{cases}
d \XXX(t) = \left [\mathcal{A}\XXX(t) + F(t,\XXX)\right ] dt + G(t,\XXX(t)) dW(t) \, ,\quad t \geq 0\, ,\\
\XXX(0) = \XXX_0 \in \EP\, ,
\end{cases}
\end{equation}
%where the notation is as already introduced. 
where 
$ F:[0,T] \times \EP \to \EP\,$, and such that
\begin{equation}\label{EQN:F}
F(t,\XXX) = \left (f(t,u),0,0\right )^T \, ,\, \text{ being }\XXX =(u,y,\eta)\, \in \EP  , \, 
\end{equation}
with
\[
\left (f(t,u)\right )(x) = \left (f_1(t,x,u_1(t,x)),\dots,f_m(t,x,u_m(t,x)) \right )^T\, .
\]
The following result provides the existence and uniqueness of a \textit{mild solution} to equation \eqref{EQN:MildLip}.

\begin{Theorem}\label{THM:E!Lip}
Let  assumptions  \ref{ASS:1}, \ref{ASS:3} and \ref{ASS:4}, hold true. Then, there exists a unique mild solution, in the sense of the Definition \ref{DEF:Mild}, to equation \eqref{EQN:MildLip}.
\end{Theorem}
\begin{proof}
It is enough to show that the map $F$ defined in equation \eqref{EQN:F} is Lipschitz continuous on the Hilbert space $\EP$. In fact, assumptions \ref{ASS:4} imply that  
\begin{equation}\label{EQN:LipF}
|F(t,\XXX) - F(t,\YYY)|_{\EP} = |f(t,u) -f(t,v)|_{X^2} \leq K |u-v|_{X^2} \leq | \XXX - \YYY |_{\EP}\, ,
\end{equation}
for any $\XXX = (u,y,\eta)^T$ and any $\YYY = (v,z,\zeta)^T \in \EP$. 
Then, exploiting  equation \eqref{EQN:LipF}, together with Proposition \ref{PRO:HSG}, the existence of a unique mild solution is a direct application of \cite[Th. 5.3.1]{DapE}, see also \cite{DPZ}.
\end{proof}

\section{Application to stochastic optimal control}\label{SEC:Opt}

The present section is mainly devoted  to the study and characterization of  the stochastic optimal  
control associated to a general non-linear system of the form

\begin{equation}\label{EQN:MildCon}
\begin{cases}
d \XXX^z(t) &= \left [\mathcal{A}\XXX^z(t) + F(t,\XXX^z)+ G(t,\XXX^z(t))R(t,\XXX^z(t),z(t))  \right ] dt +\\
&+ G(t,\XXX^z(t)) dW(t)\, ,\\
\XXX^z(t_0) &= \XXX_{0} \in \EP\, ,
\end{cases}
\end{equation}

where, besides having used the notations defined along previous sections, we denote by $z$ the control, 
while we use  the notation $\XXX^z$, to indicate the explicit dependence of the process $\XXX \in \EP$, from the control $z$. 
In what follows we exploit the results contained in \cite{FT2}, where a general characterization of stochastic optimal control problem in infinite dimension is given by means of a forward-backward-SDE approach. Therefore, the control problem  defined by equation \eqref{EQN:MildCon}, is to be understood in the weak sense, see also, e.g., \cite{DPZ, Fle}.

As stated in \cite{FT2}, we first fix $t_0 \geq 0$ and $\XXX_0 \in \EP$, then an \textit{Admissible Control System} (ACS) is given by $\mathbb{U} = \left (\Omega,\mathcal{F},\left (\mathcal{F}_t\right )_{t \geq 0},\mathbb{P},\left (W(t)\right )_{t \geq 0},z\right )$, where
\begin{itemize}
\item $\left (\Omega,\mathcal{F},\left (\mathcal{F}_t\right )_{t \geq 0},\mathbb{P}\right )$ is a complete probability space, where the filtration $\left (\mathcal{F}_t\right )_{t \geq 0}$ satisfies the usual assumptions;\\
%\item $\left (\mathcal{F}_t\right )_{t \geq 0}$ is a filtration in the above defined probability space satisfying the usual assumptions;\\
\item $\left (W(t)\right )_{t \geq 0}$ is a $\mathcal{F}_t -$adapted Wiener process taking values in $\EP$;\\
\item $z$ is a process taking values in the space $Z$, predictable with respect to the filtration $\left (\mathcal{F}_t\right )_{t \geq 0}$, and such that $z(t) \in \mathcal{Z}$ $\mathbb{P}-$a.s., for almost any $t \in [t_0,T]$, being $\mathcal{Z}$ a suitable domain of $Z$.
\end{itemize}

To each ACS, we  associate the mild solution $\XXX^z \in C([t_0,T];L^2(\Omega;\EP))$ to the abstract equation \eqref{EQN:MildCon}. Consequently, we can introduce the functional cost
\begin{equation}\label{EQN:FuncCost}
J(t_0,\XXX_0,\mathbb{U}) = \mathbb{E} \int_{t_0}^T l\left (t,\XXX^z(t),z(t)\right ) dt + \mathbb{E}\varphi(\XXX^z(T))\, ,
\end{equation}
where the function $l$, resp. $\varphi$, denotes the \textit{running cost}, resp. the \textit{terminal cost}. 
Our goal is to minimize the functional $J$ over all admissible control system. If a minimizing ACS for the functional $J$ exits, then it is called optimal control.

Throughout this section we will make use of the assumptions
 \ref{ASS:1}, \ref{ASS:3}, and \ref{ASS:4}, moreover we will also assume the following
\begin{Assumptions}\label{ASS:6}
\begin{description}
\item[(i)] the map $R:[0,T] \times \EP \times \mathcal{Z} \to \EP$ is measurable and it satisfies
\[
\begin{split}
&|R(t,\XXX,z) -R(t,\XXX,z)|_{\EP} \leq C_R(1+|\XXX|_{\EP}+|\YYY|_{\EP})^m |\XXX - \YYY |_{\EP}\, ,\\
&|R(t,\XXX,z)|_{\EP} \leq C_R\, ;
\end{split}
\]
 for some $C_R>0$ and $m \geq 0$;
\item[(ii)] the map $l:[0,T] \times \EP \times \mathcal{Z} \to \RR \cup \{+\infty\}$ is measurable and it satisfies
\[
\begin{split}
&|l(t,\XXX,z) -l(t,\XXX,z)| \leq C_l(1+|\XXX|_{\EP}+|\YYY|_{\EP})^m |\XXX - \YYY |_{\EP}\, ,\\
&|l(t,0,z)|_{\EP} \geq -C\, ,\\
& \inf_{z \in \mathcal{Z}} l(t,0,z) \leq C_l\, ;
\end{split}
\]
for some $C>0$, $C_l \geq 0$ and $m\geq 0$;
\item[(iii)] the map $\varphi: \EP \to \RR$ satisfies
\[
|\varphi(\XXX) - \varphi(\YYY)| \leq C_\varphi(1+|\XXX|_{\EP}+|\YYY|_{\EP})^m |\XXX - \YYY |_{\EP}\, .
\]
for some $C_\varphi>0$ and $m \geq 0$.
\end{description}
\end{Assumptions} 

Under assumptions \ref{ASS:1}, \ref{ASS:3}, \ref{ASS:4}, and \ref{ASS:6},  we can construct, see \cite{FT2},  an ACS as follows.

Exploiting the fact that $R$ is bounded we can therefore apply \textit{Girsanov theorem}, so that we have,  $\forall \zeta \in \mathcal{Z}$, there exists a probability measure $\mathbb{P}^\zeta$, such that
\[
W^\zeta (t) := W(t) - \int_{t_0 \wedge t}^{t \wedge T} R(s,\XXX(s),\zeta) ds\, ,
\]
is a Wiener process. Then, we may rewrite equation \eqref{EQN:MildCon} in terms of the new Wiener process $W^\zeta (t)$ and we consider the uncontrolled equation
\begin{equation}\label{EQN:MildLipCon}
\begin{cases}
d \XXX(t) = \left [\mathcal{A}\XXX(t) + F(t,\XXX)\right ] dt + G(t,\XXX(t)) dW^\zeta(t) \, ,\quad t \geq 0\, ,\\
\XXX(0) = \XXX_0 \in \EP\: ;
\end{cases}
\end{equation}
from Theorem \ref{THM:E!Lip}, we have that there exists a unique mild solution to equation \eqref{EQN:MildLipCon}.

Consequently,   $\forall t \in [0,T]$, and  $\forall (\XXX,\YYY) \in \EP \times \EP$, we  define  the Hamiltonian function related to the aforementioned problem, as follows
\begin{equation}\label{HamProblem}
\begin{split}
\psi(t,\XXX,\YYY) &:= - \inf_{z \in \mathcal{Z}} \{l(t,\XXX,z) + \YYY R(t,\XXX,z)\}\, ,\\
\Gamma(t,\XXX,\YYY) &:= \{z \in \mathcal{Z} \, : \, \psi(t,\XXX,\YYY) + l(t,\XXX,z)+\vv R(t,\XXX,z)=0\}\, ,
\end{split}
\end{equation}
where we would underline that the set $\Gamma(t,\XXX,w)$ is a (possibly empty) subset of $\mathcal{Z}$, while the function $\psi$  satisfies assumptions \ref{ASS:6}.

Within the present setting, we can apply \cite[Th. 5.1]{FT2} to write the \textit{Hamilton-Jacobi-Bellman} (HJB) equation associated to the problem stated by \eqref{EQN:MildCon} together with \eqref{EQN:FuncCost}. In particular, we have
\begin{equation}\label{EQN:HJB}
\begin{cases}
\frac{\partial w(t,\XXX)}{\partial t} + \mathcal{L}_t w(t,\XXX) = \psi(t,\XXX, \nabla w(t,\XXX) G(t,\XXX))\, ,\\
w(T,\XXX) = \varphi(\XXX)\, ,
\end{cases}
\end{equation}
where
\[
\mathcal{L}_t w (\XXX) := \frac{1}{2} Tr\left [G(t,\XXX) G(t,\XXX)^* \nabla^2 w(\XXX) \right ] + \left \langle \mathcal{A}\XXX,\nabla w(\XXX) \right \rangle_{\EP}\, ,
\]
is the infinitesimal generator of the equation \eqref{EQN:MildCon}, while $Tr$ stands for the {\it trace}, and $G^*$ is the adjoint of $G$.

%We will say that $w$ is a \textit{mild solution in the sense of generalized gradient}, mild solution for short in what follows, if the following holds, see, e.g., \cite[Def. 5.1]{FT2}.
In what follows we exploit the following definition, see, e.g.,  \cite[Def. 5.1]{FT2}.
\begin{Definition}\label{DEF:Mild}
A function $u : [0,T] \times \XX \to \RR$ is defined to be a mild solution in the sense of generalized gradient, to equation \eqref{EQN:HJB} if the following hold:
\begin{description}
\item[(i)] there exists $C>0$ and $m \geq 0$ such that for any $t \in [0,T]$ and any $\uu$, $\vv \in \XX$ it holds
\[
\begin{split}
&|w(t,\XXX)-w(t,\YYY)| \leq C (1+ |\XXX|_{\EP}+|\YYY|_{\EP})^m|\XXX-\YYY|_{\EP}\, ,\\
&|w(t,0)| \leq C\, ;\\
\end{split}
\]
\item[(ii)] for any $0 \leq t \leq T$ and $\XXX \in \EP$, we have that
\[
w(t,\XXX) = P_{t,T} \varphi(\XXX) - \int_t^T P_{t,s} \psi(s,\cdot,w(s,\cdot),\rho(s,\cdot))(\XXX) ds\, ,
\]
where $\rho$ is an arbitrary element of the \textit{generalized directional gradient} $\nabla^G w$, as it has been  defined in \cite{FT2}, while $P_{t,T}$ is the Markov semigroup generated by the forward process \eqref{EQN:MildCon}.
\end{description}
\end{Definition}

\begin{Remark}\label{RemarkGradient}
We would like to underline that, following the approach developed in \cite{FT2}, we do not need to require any differentiability properties for the function $F$, $G$ and $w$. In fact, the notion of {\it gradient} appearing in equation \eqref{EQN:HJB}, is to be understood in a weak sense, namely in terms of the \textit{generalized directional gradient}. In fact, in \cite{FT2} the authors show that, if $w$ is regular enough, then $\nabla w$ coincides with the standard notion of gradient. The latter implies that, in the present case, the \textit{generalized directional gradient}   coincides with the \textit{Fr\'{e}chet derivative}, resp. with the \textit{G\^{a}teaux derivative}, if we assume $w$ to be \textit{Fr\'{e}chet} differentiable, resp. to be \textit{G\^{a}teaux} differentiable.
\end{Remark}

In the light of Definition \ref{DEF:Mild} and Remark \ref{RemarkGradient}, we have the following.

\begin{Proposition}\label{PRO:FT2}
Let us consider the optimal control problem defined by \eqref{EQN:MildCon} and \eqref{EQN:FuncCost}, then the equation \eqref{EQN:HJB} provides the associated HJB problem. Moreover, if  assumptions \ref{ASS:1}, \ref{ASS:3}, \ref{ASS:4}, and \ref{ASS:6} hold true, then we have that the HJB equation \eqref{EQN:HJB} admits a unique mild solution, in the sense of the definition \ref{DEF:Mild}.
\end{Proposition}
\begin{proof}
The proof immediately follows exploiting  \cite[Th. 5.1]{FT2}.
\end{proof}

As a direct consequence of Proposition \ref{PRO:FT2}, we provide a {\it synthesis} of the optimal control problem, by the following 

\begin{Theorem}\label{THM:Synth}
Let  assumptions \ref{ASS:1}, \ref{ASS:3}, \ref{ASS:4}, and \ref{ASS:6} hold true. Let $w$ be a mild solution to the HJB equation \eqref{EQN:HJB}, and chose $\rho$ to be an element of the generalized directional gradient $\nabla^G w$. Then, for all ACS, we have that $J(t_0,\XXX_0,\mathbb{U}) \geq w(t_0,\XXX_0)$, and the equality holds if and only if the following feedback law is satisfied by $z$ and $\uu^z$
\begin{equation}\label{EQN:ACSOpt}
z(t) = \Gamma \left (t,\XXX^z(t),G(t,\rho(t,\XXX^z(t))\right )\, ,\quad \mathbb{P}-\,a.s. \, \mbox{ for a.a. } \, t \in [t_0,T]\, .
\end{equation}
Moreover, if there exists a measurable function $\gamma:[0,T] \times \EP \times \EP \to \mathcal{Z}$ with
\[
\gamma(t,\XXX,\YYY) \in \Gamma(t,\XXX,\YYY)\, ,\quad t \in [0,T] \, ,\, \XXX\, , \, \YYY \in \XX\, ,
\]
then there also exists, at least one ACS such that
\[
\bar z(t) = \gamma(t,\XXX^z(t),\rho(t,\XXX^z(t)))\, ,\quad \mathbb{P}-\,a.s. \, \mbox{ for a.a. } \, t \in [t_0,T]\, ,
\]
where $\XXX^{\bar z}$ is a mild solution to equation 
\begin{equation}\label{EQN:MildConLoop}
\begin{cases}
d \XXX^z(t) &= \left [\mathcal{A}\XXX^z(t) + F(t,\XXX^z)\right ]dt +\\
&+\left [G(t,\XXX^z(t))R(t,\XXX^z(t),\gamma(t,\XXX^z(t),\rho(t,\XXX^z(t))))  \right ] dt +\\
&+ G(t,\XXX^z(t)) dW(t)\, ,\\
\XXX^z(t_0) &= \XXX_{0} \in \EP\, ,
\end{cases}
\end{equation}
\end{Theorem}
\begin{proof}
See \cite[Th. 7.2]{FT2}.
\end{proof}

\begin{Example}[The heat equation with controlled stochastic boundary conditions on a graph]

In what follows we model the heat equation over a finite graph  $\mathbb{G}$, considering local controlled dynamic boundary conditions,  namely, see \ref{SEC:Intro}, we have a total of $m$ nodes, and $n_0 =n$ nodes equipped with dynamic boundary conditions. We also assume that there is not a noise affecting  the heat equation, whereas we assume the boundary condition to be perturbed by an additive Wiener process. 
Summing up, by means of the notations introduced along previous sections, we deal with the following  system

{\footnotesize
\begin{equation}\label{EQN:HeatEq}
\begin{cases}
\dot{u}_j (t,x) = \left (c_j u'_i\right )' (t,x) \, , \quad t \geq 0\, ,\, x \in (0,1)\, , \, j = 1,\dots,m\, ,\\[1.5ex]
u_j(t,\mathrm{v}_\alpha) = u_l(t,\mathrm{v}_\alpha) =: d^\alpha(t)\, , \quad t \geq 0\, ,\, l,\, j \in \Gamma(\mathrm{v}_i)\, , \, j = 1,\dots,m\, ,\\[1.5ex]
\dot{d}^\alpha (t) = - \sum_{j=1}^m \phi_{\alpha,j} c_j(\mathrm{v}_\alpha) u'_j(t,\mathrm{v}_\alpha) + \frac{1}{T}\int_{-T}^0 d^\alpha(t+\theta)d\theta + \tilde{g}_\alpha(t)\left ( z(t) + \dot{W}^2_\alpha (t)\right ) \, ,\\
\qquad \qquad \qquad \qquad \qquad \qquad \qquad \qquad \qquad \qquad \qquad \qquad \quad t \geq 0 \, ,\, \alpha=1,\dots,n\, ,\\[1.5ex]
u_j(0,x) = u^0_j(x) \, ,\quad x \in (0,1)\, ,\, j=1,\dots,m\, ,\\[1.5ex]
d^\alpha(0) = d_\alpha^0\, ,\quad \alpha=1,\dots,n\, .\\
\end{cases}
\end{equation}
}
Then, we rewrite the system \eqref{EQN:HeatEq}, as an abstract Cauchy problem on the Hilbert space $\XX$, as follows
\begin{equation}\label{EQN:MildConHeat}
\begin{cases}
d \XXX(t)^z = \mathcal{A}\XXX^z(t) dt + G(t,\XXX^z(t))\left (R z(t) + dW(t)\right ) \, ,\quad t \in [t_0,T]\, ,\\
\XXX^z(t_0) = \XXX_{0} \in \EP\, ,
\end{cases}
\end{equation}
where $R: \RR^n \to \EP$ is the immersion of the boundary space $\RR^n$ into the product space $\EP$. In the present setting the control $z$ takes values in $\RR^n$, while  $\mathcal{Z}$ is a subset of $\RR^n$. Considering a cost functional of the form \eqref{EQN:FuncCost}, then Proposition \ref{PRO:FT2} together with  Theorem \ref{THM:Synth}, imply  the existence of, at least, one ACS for the HJB equation \eqref{EQN:HJB} associated with the stochastic control problem \eqref{EQN:MildConHeat}-\eqref{EQN:FuncCost}.
Consequently, the synthesis of the optimal control problem, reads as follows
\end{Example}

\begin{Theorem}\label{THM:Synth}
Let  assumptions \ref{ASS:1}, \ref{ASS:3}, \ref{ASS:4}, and \ref{ASS:6} hold true. 
Let $w$ be a mild solution to the HJB equation \eqref{EQN:HJB}, and choose $\rho$ to be an element of the generalized directional gradient $\nabla^G w$.
Then, for all ACS, we have that $J(t_0,\XXX_0,\mathbb{U}) \geq w(t_0,\XXX_0)$, and the equality holds if and only of the following feedback law is satisfied by $z$ and $\XXX^z$
\begin{equation}\label{EQN:ACSOpt}
z(t) = \Gamma \left (t,\XXX^z(t),G(t,\rho(t,\XXX^z(t))\right )\, ,\quad \mathbb{P}-\,a.s. \, \mbox{ for a.a. } \, t \in [t_0,T]\, .
\end{equation}
Moreover, if there exists a measurable function $\gamma:[0,T] \times \EP \times \EP \to \mathcal{Z}$ with
\[
\gamma(t,\XXX,\YYY) \in \Gamma(t,\XXX,\YYY)\, ,\quad t \in [0,T] \, ,\, \XXX\, , \, \YYY \in \EP\, ,
\]
then there also exists at least one ACS, such that
\[
\bar z(t) = \gamma(t,\XXX^z(t),\rho(t,\XXX^z(t)))\, ,\quad \mathbb{P}-\,a.s. \, \mbox{ for a.a. } \, t \in [t_0,T]\, .
\]
Eventually, we have that $\XXX^{\bar z}$ is a mild solution to equation 
\begin{equation}\label{EQN:MildConHeat}
\begin{cases}
d \XXX(t)^z = \mathcal{A}\XXX^z(t) dt + G(t,\XXX^z(t))\left (R \gamma(t,\XXX^z(t),\rho(t,\XXX^z(t))) + dW(t)\right ) \, ,\quad t \in [t_0,T]\, ,\\
\XXX^z(t_0) = \XXX_{0} \in \EP\, .
\end{cases}
\end{equation}
\end{Theorem}


\newpage
\begin{thebibliography}{}
\addcontentsline{toc}{chapter}{Bibliografy}
\bibliographystyle{plain}
\thispagestyle{empty}

\bibitem{ALbDP} S. Albeverio and L. Di Persio, "Some stochastic dynamical models in neurobiology: recent developments", European Communications in Mathematical and Theoretical Biology, : 44-53, (2011).

\bibitem{Allen} L. JS. Allen, "An introduction to stochastic processes with applications to biology" CRC Press, (2010).

\bibitem{BarbuCorDP} V. Barbu, F. Cordoni and L. Di Persio, "Optimal control of stochastic FitzHugh-Nagumo equation", International Journal of Control, (\textbf{4}), 746–756, (2016).

\bibitem{Bat1} A. Bátkai and S. Piazzera, Semigroups for delay equations. Wellesley: AK Peters, (2005).

\bibitem{Bat2} A. Bátkai and S. Piazzera, "Semigroups and linear partial differential equations with delay." Journal of mathematical analysis and applications 264.1 : 1-20, (2001).

\bibitem{BenazzoliDP} C. Benazzoli and L. Di persio "Default contagion in financial networks." International Journal of Mathematics and Computers in Simulation, Vol. 10, Pages 112-117, (2016)

\bibitem{BCM} S. Bonaccorsi, F. Confortola and E. Mastrogiacomo. "Optimal control of stochastic differential equations with dynamical boundary conditions." Journal of Mathematical Analysis and Applications 344.2: 667-681, (2008).

\bibitem{BZM} S. Bonaccorsi, C. Marinelli and G. Ziglio, "Stochastic FitzHugh-Nagumo equations on networks with impulsive noise." Electron. J. Probab 13: 1362-1379, (2008).

\bibitem{BZ} S. Bonaccorsi and G. Ziglio. "A semigroup approach to stochastic dynamical boundary value problems." Systems, control, modeling and optimization. Springer US,  55-65, 2006.

\bibitem{Car} S. Cardanobile and D. Mugnolo, "Analysis of a FitzHugh-Nagumo-Rall model of a neuronal network.", Mathematical Methods in the Applied Sciences 30.18: 2281-2308, (2007).

\bibitem{Car3} S. Cardanobile and D. Mugnolo, "Qualitative properties of coupled parabolic systems of evolution equations.", Annali della Scuola Normale Superiore di Pisa–Classe di Scienze 7.2 : 287-312, (2008).

\bibitem{Con} F. Confortola and E. Mastrogiacomo, "Optimal control for stochastic heat equation with memory", Evol. Equ. Control Theory 3, (\textbf{1}), 35–58,  (2014).

\bibitem{CDP} F. Cordoni and L. Di Persio, "Gaussian estimates on networks with dynamic stochastic boundary conditions", Infinite Dimensional Analysis and Quantum Probability, Vol. 20 (1), (2017).


\bibitem{CDPBackward1} F. Cordoni and L. Di Persio, "Backward stochastic differential equations approach to hedging, option pricing, and insurance problems", International Journal of Stochastic Analysis,
Vol. 2014, Art. No. 152389, (2014)

\bibitem{CDPBackward2} F. Cordoni and L. Di Persio, "A bsde with delayed generator approach to pricing under counterparty risk and collateralization", International Journal of Stochastic Analysis,
Vol. 2016, Art. No. 1059303, (2016).

\bibitem{Dap} G. Da Prato and J. Zabczyk, "Stochastic equations in infinite dimensions", Vol. 152. Cambridge university press, (2014).

\bibitem{DapE} G. Da Prato and J. Zabczyk. "Ergodicity for infinite dimensional systems", Vol. 229. Cambridge University Press, (1996).

\bibitem{DPBonolloPellegrini} L. Di Persio and M. Bonollo and G. Pellegrini , "Polynomial chaos expansion approach to interest rate models",  Journal of Probability and Statistics, Vol. 2015, Art. No. 369053 (2015)

\bibitem{DPZ} L. Di Persio and G. Ziglio, "Gaussian estimates on networks with applications to optimal control", Networks Het. Media, (\textbf{6}), 279--296, (2011).

\bibitem{Eng1} K.J. Engel, "Spectral theory and generator property for one-sided coupled operator matrices", Semigroup Forum. Vol. 58. No. 2. Springer-Verlag New York., (1999).

\bibitem{Eng} K. J. Engel and R. Nagel, "One-parameter semigroups for linear evolution equations", Vol. 194. Springer Science $\&$ Business Media, (2000).

\bibitem{Fle} W. H. Fleming and H. M. Soner, Controlled Markov processes and viscosity solutions. Vol. 25. Springer Science $\&$ Business Media, (2006).

\bibitem{FT2} M. Fuhrman and G. Tessitore, "Generalized directional gradients, backward stochastic differential equations and mild solutions of semilinear parabolic equations", Applied Mathematics and Optimization 51.3: 279-332, (2005).

\bibitem{GarZoll} C. Gardiner and P. Zoller, "Quantum noise: a handbook of Markovian and non-Markovian quantum stochastic methods with applications to quantum optics", Vol. 56. Springer Science \& Business Media, (2004).

\bibitem{Gol} G. R. Goldstein, "Derivation and physical interpretation of general boundary conditions." Advances in Differential Equations 11.4: 457-480, (2006).

\bibitem{HiSa} A. Hirsa and S. N. Neftci "An introduction to the mathematics of financial derivatives", Academic Press, (2013).

\bibitem{Inf} G. Infante and J. R. L. Webb, "Nonlinear non-local boundary-value problems and perturbed Hammerstein integral equations." Proceedings of the Edinburgh Mathematical Society (Series 2) 49.03: 637-656,  (2006).

\bibitem{Inf1} G. Infante and J. R. L. Webb,. "Loss of positivity in a nonlinear scalar heat equation." Nonlinear Differential Equations and Applications NoDEA 13.2: 249-261, (2006).

\bibitem{Inf2} G. Infante and P. Pietramala, "A third order boundary value problem subject to nonlinear boundary conditions." Mathematica Bohemica 135.2: 113-121,  (2010).

\bibitem{Kai} C. Kaiser, "Integrated semigroups and linear partial differential equations with delay." Journal of mathematical analysis and applications 292.2 (2004): 328-339.

\bibitem{Kleinert} H. Kleinert, "Path integrals in quantum mechanics, statistics, polymer physics, and financial markets", World Scientific (2009).

\bibitem{Mug} D. Mugnolo, "Gaussian estimates for a heat equation on a network." Netw. Heterog. Media 2 (2007), (\textbf{1}), 55–79..

\bibitem{MugB} D. Mugnolo, Semigroup methods for evolution equations on networks. Springer, (2014).

\bibitem{MugR} D. Mugnolo and S. Romanelli, "Dynamic and generalized Wentzell node conditions for network equations." Mathematical methods in the applied sciences 30.6 (2007): 681-706.

\bibitem{Oua} E. M. Ouhabaz, Analysis of Heat Equations on Domains.(LMS-31). Princeton University Press, (2009).

\bibitem{Tum} R. Tumulka, "The analogue of Bohm–Bell processes on a graph." Physics Letters A 348.3: 126-134, (2006).

\bibitem{Web} J. R. L. Webb, "Optimal constants in a nonlocal boundary value problem." Nonlinear Analysis: Theory, Methods $\&$ Applications 63.5: 672-685, (2005).

\end{thebibliography}
\end{document}